\documentclass[12pt]{amsart}

\usepackage{epsfig}
\usepackage{amsfonts}
\usepackage{amssymb}
\usepackage{amsmath}
\usepackage{amsthm}
\usepackage{amscd}
\usepackage{amstext}
\usepackage{latexsym}

\usepackage{fancyhdr}

\usepackage{enumerate}

\newcommand{\mfa}{\mathfrak{a}}

\newcommand{\mfm}{\mathfrak{m}}

\newcommand{\CC}{\mathbf{C}}

\newcommand{\RR}{\mathbf{R}}
\newcommand{\DD}{\mathbf{D}}

\newcommand{\JJ}{\mathcal{J}}

\newcommand{\OO}{\mathcal{O}}

\newcommand{\FF}{\mathcal F}

\newcommand{\EE}{\mathcal E}

\newcommand{\EEt}{\tilde{\mathcal E}}

\newcommand{\Gm}{\Gamma}

\newcommand{\Om}{\Omega}

\newcommand{\ld}{\lambda}

\newcommand{\qa}{\quad}
\newcommand{\vp}{\varphi}

\newcommand{\ep}{\epsilon}

\newcommand{\noi}{\noindent}

\providecommand{\abs}[1]{\left|#1\right|}

\providecommand{\inp}[1]{\langle#1\rangle}

\theoremstyle{plain}

\newtheorem{theorem}{Theorem}[section]

\newtheorem{lemma}[theorem]{Lemma}
\newtheorem{corollary}[theorem]{Corollary}
\newtheorem{proposition}[theorem]{Proposition}

\newtheorem{conjecture}[theorem]{Conjecture}

 \newtheorem{example}[theorem]{\textnormal{\textbf{Example}}}

\newtheorem{remark}[theorem]{Remark}

\theoremstyle{remark}
\newtheorem{remark1}[theorem]{Remark}

\DeclareMathOperator{\Covol}{Covol}

\DeclareMathOperator{\Vol}{Vol}

\DeclareMathOperator{\interior}{int}

\DeclareMathOperator{\PSH}{PSH}

\DeclareMathOperator{\loc}{loc}

\begin{document}

\title[Higher Lelong numbers and convex geometry]{Higher Lelong numbers
and convex geometry}

\author{Dano Kim and Alexander Rashkovskii}

\date{}

\begin{abstract}

\noindent We prove  the reversed Alexandrov-Fenchel inequality for mixed Monge-Amp\`ere masses of plurisubharmonic functions, which generalizes a result of Demailly and Pham. As applications to convex geometry, this gives a complex analytic proof of the reversed Alexandrov-Fenchel inequality for mixed covolumes, which generalizes recent results in convex geometry of Kaveh-Khovanskii, Khovanskii-Timorin, Milman-Rotem and R. Schneider on reversed (or complemented) Brunn-Minkowski and Alexandrov-Fenchel inequalities. Also for toric plurisubharmonic functions in the Cegrell class, we confirm Demailly's conjecture on the convergence of higher Lelong numbers under the canonical approximation.

\end{abstract}

\maketitle

\section{Introduction}

 Let $\vp$ be a plurisubharmonic ({psh}, for short) function with an isolated singularity at a point $p \in \Omega \subset \CC^n$, i.e. $\vp$ is locally bounded in $\Omega \setminus \{ p \}$. Then higher Lelong numbers $\ell_2(\vp, p), \ldots, \ell_n (\vp, p)$ are defined for $\vp$ by (see Section~\ref{cobdd} : cf. \cite{D93}, \cite{DX}, \cite{S85})

 $$ \ell_k (\vp, p) = \int_{\{p\}} (dd^c \vp)^k \wedge (dd^c \log \abs{z} )^{n-k} $$

\noindent in addition to its usual Lelong number which we denote by $\ell_1(\vp, p)$. When $\vp$ is a psh function associated to an $\mathfrak{m}$-primary ideal $\mfa$ (where $\mathfrak{m}$ is the maximal ideal of the point $p$), i.e. $\vp = \frac{1}{2} \sum^N_{j=1} \abs{g_j}^2$ where $g_1, \ldots, g_N$ are generators of $\mfa$, the $n$-th Lelong number $\ell_n (\vp, p)$ recovers the well known Hilbert-Samuel multiplicity of $\mfa$ (see \cite[Lemma 2.1]{D09}).  Furthermore,  all the higher Lelong numbers are well-defined on functions from a class of psh functions generalizing those with isolated singularity, \cite{C04}. It is the Cegrell class $\EE$ (see Section~\ref{cobdd}), which is the largest class of negative psh functions $\vp$ with well-defined Monge-Amp\`ere operator $(dd^c \vp)^n$.

   By Demailly's theorem \cite{D92}, any psh function $\vp$ can be approximated by psh functions $\vp_m$ with analytic singularities such that $ \ell_1 (\vp_m, 0) \to \ell_1 (\vp, 0) $ as $m \to \infty$.

  Regarding the nature of the higher Lelong numbers and its analytic intersection theory, the most fundamental question would be the following conjecture which sits at the intersection of complex analysis, singularity theory and algebraic geometry.

\begin{conjecture}[See \cite{D12}, \cite{R16}, \cite{DGZ}]\label{dem}

Let $\vp$ be a plurisubharmonic function in the Cegrell class $\EE = \EE(B)$ in the unit ball $B$ of $\CC^n$.

(a) (Demailly)  If $\{ \vp_m \}_{m \ge 1}$ is the Demailly approximation sequence of $\vp$ \cite{D92}, then we have the convergence of $k$-th Lelong numbers for   $2 \le k \le n$:

 $$ \ell_k (\vp_m, 0) \to \ell_k (\vp, 0) $$ as $m$ goes to infinity.

(b) (Guedj, Rashkovskii)  If $\ell_n (\vp, 0) = 0$, then $\ell_1 (\vp, 0)=0$.

\end{conjecture}

It is easy to see that (a) implies (b). The singularity of the Demailly approximation function $\vp_m$ of a psh function $\vp$ is given by the multiplier ideal $\JJ(m \vp)$. Thus $\ell_k (\vp_m, 0)$ is equal to   $m^{-k}$ times the $k$-th mixed multiplicity (generalizing the Hilbert-Samuel multiplicity) of $\JJ(m \vp)$, and the validity of this conjecture will provide algebro-geometric contents of the higher Lelong numbers.  See e.g. \cite[p.107]{R16} and the references therein (including \cite{BFJ08}) for previous results on Conjecture~\ref{dem}.

 A psh function $\vp(z_1, \ldots, z_n)$ is called \emph{toric} (or \emph{multi-circled}) if the value depends only on $\abs{z_1}, \ldots, \abs{z_n}$, i.e. $\vp(z_1, \ldots, z_n) = \vp (\abs{z_1}, \ldots, \abs{z_n})$. (See e.g. \cite{G} and \cite{R13b} for more on toric psh functions.) Toric psh functions provide many useful examples in the study of psh singularities. Also some key arguments of important related results (such as in \cite{DH}, \cite{M02}) are made by reducing to the toric case. In this article, we first have the following result.

 \begin{theorem}\label{dem toric}

 Conjecture~\ref{dem} is true if $\vp$ is toric.

 \end{theorem}

 A special case of Theorem~\ref{dem toric} when $\vp$ is further assumed (in addition to being in the Cegrell class $\EE$) to be locally bounded outside $0$ was first proved in \cite{R13a} using the notion of asymptotically analytic singularity. In general, a toric psh function in the Cegrell class need not be asymptotically analytic \cite[Example 3.9]{R13a}   unless it is locally bounded outside $0$.
 
   To our knowledge, in the generality of Cegrell psh functions (i.e. outside the class of psh functions with isolated singularities), Theorem~\ref{dem toric} is the first general result toward Conjecture~\ref{dem}. Note that the singularities of a psh function in the Cegrell classes are typically even more complicated than the case of a psh function with isolated singularities (see Example~\ref{limit}). 

\begin{remark}\label{rr0}
 
  Both Conjecture~\ref{dem} and Theorem~\ref{dem toric} can be generalized to a version for general mixed Monge-Amp\`ere masses:  $m(\vp_{1,m}, \ldots, \vp_{n,m}) \to m(\vp_1, \ldots, \vp_n)$ where $\{ \vp_{j,m} \}_{m \ge 1}$ is the Demailly approximation sequence of $\vp_{j}$ where every $\vp_{j} \; (j = 1, \ldots, n)$ is toric and in the Cegrell class $\EE$. In this paper, we will concentrate on the original version of the problem, but see Remarks~\ref{rr1} and \ref{rr2}. 
 
\end{remark}

 The proof of Theorem~\ref{dem toric} is entirely different from \cite{R13a}:  it uses the multiplier ideals and works for general toric $\vp$ in the Cegrell class $\EE$. In the proof, we consider relations among mixed multiplicities $e(I_1, \ldots, I_n)$ of $\mfm$-primary ideals, mixed covolumes $\Covol(\Gamma_1, \ldots, \Gamma_n)$ of convex regions in $\RR_+^n$, and mixed Monge-Amp\`ere masses at the origin $m(\vp_1, \ldots, \vp_n)$ of Cegrell psh functions  (see Section 2).   The basic theories of these three kinds of objects are parallel to each other and all analogous to the classical theory of mixed volumes of convex bodies: their definitions given by the homogeneous polynomial theorem, the polarization identity, Brunn-Minkowski inequalities and so on (see \cite[Chap.4]{BZ}, \cite{G02} and Section 2).

 This sheds light on the following natural question: what are the relations between this convex geometry and the theory of higher Lelong numbers (or more generally, mixed Monge-Amp\`ere masses)? This question is motivated by one important instance of such relation given by Demailly and Pham~\cite[Lemma 2.1]{DH}, which implied a proof of the $1 < k < n$ version of Conjecture~\ref{dem} (b). Indeed, such a question was raised explicitly e.g. by Florin Ambro at the end of an inspiring lecture of Jean-Pierre Demailly on \cite{DH}, \cite{D12} at LMS Durham Symposium in July 2012.

  In the second part of this article, we answer the question by showing that the  mixed Monge-Amp\`ere masses of Cegrell psh functions generalize, in a specific sense, both mixed covolumes and mixed multiplicities. As a consequence, we give a unified analytic proof of the reversed Alexandrov-Fenchel inequalities for all three of them. 

 The Alexandrov-Fenchel inequality is a significant generalization of the Brunn-Minkowski inequality both in the classical setting of convex bodies $\&$ mixed volumes and in the current setting of `coconvex' bodies $\&$ mixed covolumes as in Khovanskii-Timorin~\cite{KT14} and Kaveh-Khovanskii~\cite{KK13}. According to \cite[p.811]{KT14}, the reversed versions of Brunn-Minkowski, First Minkowski, Second Minkowski inequalities can be derived from the reversed Alexandrov-Fenchel inequality as in the corollary below of the following theorem.

 \begin{theorem}[Alexandrov-Fenchel inequality for mixed Monge-Amp\`ere mass]\label{af}

 Let $\Omega \subset \CC^n$ be a bounded hyperconvex domain for $n \ge 2$. Let $p \in \Omega$ be a point. Let $\vp_1, \ldots, \vp_n$ be psh functions in the Cegrell class ${\EE}= \EE(\Omega)$. Then we have the inequality
  $$ m(\vp_1, \vp_1, \vp_3, \ldots, \vp_n) \, m(\vp_2, \vp_2, \vp_3, \ldots, \vp_n) \ge m(\vp_1, \vp_2, \vp_3, \ldots, \vp_n)^2    $$

\noi where $m(\vp_1, \vp_2, \ldots, \vp_n)$ is the mass of $(dd^c \vp_1) \wedge \ldots \wedge (dd^c \vp_n)$ carried by $p$.

 \end{theorem}

\noindent This generalizes the important result \cite[Lemma 2.1]{DH}. Now for mixed covolumes, we obtain

 \begin{corollary}[Alexandrov-Fenchel inequality for mixed covolumes]\label{AFM}
\qa 

  Let $n \ge 2$. Let $\; C:= \overline{\RR^n_{+}}$ be  the closed convex cone given as the closure of $\RR^n_{+} := \{ (x_1, \ldots, x_n) : x_i > 0 ,\forall i \}$.   
  Let $\Gm_1, \ldots, \Gm_n \subset C$ be cofinite $C$-convex regions. We have the inequality $$ \Covol(\Gamma_1, \Gamma_1, \Gamma_3, \ldots, \Gamma_n) \, \Covol(\Gamma_2,\Gm_2, \Gm_3, \ldots, \Gamma_n)  \ge \Covol(\Gamma_1, \Gm_2, \Gm_3,  \ldots, \Gamma_n)^2  .$$

 \end{corollary}

 We remark that Corollary~\ref{AFM} also recovers Alexandrov-Fenchel inequality for mixed multiplicities (see \cite{T77}, \cite{RS}) as in  \cite[Theorem 1]{KK13} in the nonsingular case, i.e. when the local ring is that of a nonsingular point of a complex variety.

 Note that Corollary~\ref{AFM} generalizes the Alexandrov-Fenchel inequality for mixed covolumes given in the case when $\Gm_1, \ldots, \Gm_n$ are furthermore \emph{cobounded} (see Section 2.3 for cobounded and cofinite convex regions) by  \cite{KT14} and \cite[Theorem 10.5]{KK13}. We are naturally led to consider the more general cofinite case in convex geometry from considering Cegrell psh functions in complex analysis via their indicator diagrams (see Theorem~\ref{r13} and Lemma~\ref{indic}).

  We remark that \cite{KT14} proved the cobounded case of Corollary~\ref{AFM} from the classical \emph{unreversed} Alexandrov-Fenchel inequality for convex bodies. On the other hand, \cite{KK13} proved it using the reversed Alexandrov-Fenchel inequality for mixed multiplicities  \cite[Theorem 9.1]{KK13}, which in turn relies on the classical Hodge index theorem. See \cite[Remark 1.6.8]{L} for an insightful remark on how one can view both of these two directions (reversed and unreversed) of inequalities as coming from the Hodge index theorem.

    It would be also interesting to know whether one can get an alternative proof of the unreversed Alexandrov-Fenchel inequality for convex bodies starting from the reversed one. According to V. Timorin's answer to our question, such an argument is currently not known in dimensions greater than $2$.


 After the first version of this paper was posted on arXiv, we learned that recently published papers of Rolf Schneider~\cite{S18} and Milman and Rotem~\cite{MR14} in convex geometry (in addition to \cite{KT14}, \cite{KK13})  also studied the reversed (or \emph{complemented}) version of the Brunn-Minkowski inequality and related topics. In particular, \cite[Theorem 1]{S18} considers the more general cofinite case (together with the equality condition) discussed above, similarly to this paper.

  We note that this paper is giving the more general (than Brunn-Minkowski) Alexandrov-Fenchel inequality in the cofinite case whereas our complex analytic proof is considerably short.  While our original interest comes more from complex analysis and algebraic geometry, it will be certainly interesting also to pursue more of these relations to convex geometry. 

This article is organized as follows. In Sections 2 and 3, we review the necessary background on convex geometry and on psh functions, respectively. In Section 4, we give the proof of Theorem~\ref{dem toric} on the convergence of higher Lelong numbers under Demailly approximation for toric psh functions in the Cegrell class. In Section 5, we give applications to convex geometry of mixed covolumes. In particular, a unified analytic proof is given to the reversed Alexandrov-Fenchel inequalities.

\medskip

\noi \textbf{Acknowledgements.} We would like to thank Jean-Pierre Demailly and S\'ebastien Boucksom for interesting discussions regarding Conjecture~\ref{dem}, and Kiumars Kaveh, Askold Khovanskii and Vladlen Timorin for answering our questions on their papers. We thank the anonymous referees for helpful comments.  This work began when A.R. visited Seoul National University and D.K. visited University of Stavanger and \'Ecole Polytechnique. We thank these institutions for hospitality and also SRC-GAIA {\small (Center for Geometry and its Applications, based at POSTECH, Korea)}  for its financial support for these visits through the  National Research Foundation of Korea grant No.2011-0030795.

\section{Preliminaries on convex geometry}

 In this section, for the reader's convenience, we briefly recall the background of the fundamental notions of mixed volumes, mixed multiplicities and mixed covolumes which are developed analogously to each other. This analogy will also be shared by mixed Monge-Amp\`ere mass in the next section. The material in this section is mostly classical (except recent developments on mixed covolumes) and the paper \cite{KK} has excellent exposition, which we follow.

 \subsection{Mixed volumes}
 We briefly summarize from  \cite[Sec.2]{KK} to which we refer for more details. See also \cite[Chap.4]{BZ}, \cite[Sec.2.3]{H94}, \cite{KK12}, \cite{S13}.
 Consider the cone of all convex bodies in $\RR^n$ under the Minkowski sum of two convex bodies : $\Delta_1 + \Delta_2 := \{ x + y : x \in \Delta_1, y \in \Delta_2 \}$.   For convex bodies $\Delta_1, \ldots, \Delta_n$, the volume (with respect to the standard Euclidean metric in $\RR^n$) of the convex body $\ld_1 \Delta_1 + \ldots + \ld_1 \Delta_n$ is a homogeneous polynomial in $\ld_1, \ldots, \ld_n$ and the coefficient of the term $\ld_1 \ldots \ld_n$ divided by $n!$ is called the \emph{mixed volume} $V(\Delta_1, \ldots, \Delta_n)$. It can also be regarded as the \emph{polarization} of the volume homogeneous polynomial in the sense that \cite[p.272]{KK} the mixed volume $V(\Delta_1, \ldots, \Delta_n)$ is the unique multilinear symmetric function on the cone of convex bodies such that $V(\Delta, \ldots, \Delta) = \Vol (\Delta)$. It also satisfies the following polarization identity (see e.g. \cite[p.137]{BZ}) for $\mu = V$, thus expressing the mixed volume $V$ in terms of the usual volumes $\Vol$'s :

\begin{equation}\label{polari}
  \mu (s_1, \ldots, s_n) = \frac{(-1)^n}{n!} \sum^n_{j=1} \sum_{i_1< \ldots <i_j} \mu \left(   \sum^j_{k=1}  s_{i_k}, \ldots,  \sum^j_{k=1}  s_{i_k}   \right).
\end{equation}

 The Alexandrov-Fenchel inequality for convex bodies is a significant generalization of the Brunn-Minkowski inequality for convex bodies. See \cite{G02}, \cite{Gr90} for nice surveys of these inequalities. See also \cite{Be17}, \cite{W17} for some other recent connections to complex geometry based on H\"ormander $L^2$ estimates. 

\subsection{Mixed multiplicities}

  Let $X$ be an affine complex variety and fix a point $x \in X$ with the maximal ideal $\mfm := \mfm_x$. Let $\mfa_1, \ldots, \mfa_n$ be ideal sheaves in $\OO_X$ such that they are $\mfm$-primary, i.e.  the zero set of each $\mfa_k$ is $\{ x \}$.  (For our purpose of defining multiplicities at $x$, we could instead let $\mfa_1, \ldots, \mfa_n$ be $\mfm$-primary ideals in the local ring $\OO_{X, x}$.)

  The mixed multiplicity $e(\mfa_1, \ldots, \mfa_n)$ is equal to the usual intersection multiplicity at $x$ of the generic hypersurfaces $H_k = \{ y \in X : f_k (y) = 0  \}  \; (1 \le k \le n)$ (i.e. $f_k$ is a generic element of $\mfa_k$ ). It can also be regarded as the polarization of the classical Hilbert-Samuel multiplicity (see e.g. \cite[(2.4.31)]{L}).

 See \cite[1.6.B]{L}, \cite[p.271]{KK}, \cite{KK13} (and the references therein) and the original references \cite{T73}, \cite{RS} for more on mixed multiplicites.

\subsection{Mixed covolumes}

 We will follow \cite{KK}, \cite{R13b}. Let $C$ be the closure $\overline{\RR^n_{+}}$ (of $\RR^n_{+}$) as a closed convex cone.  A closed and convex subset $\Gm \subset C$ is called a \emph{$C$-convex region} if $x \in \Gm, y \in C$ implies $x + y \in \Gm$. We call a $C$-convex region $\Gm$ \emph{cofinite} (resp. \emph{cobounded}) if the complement $C \setminus \Gm$ has finite volume (resp. is bounded). We call the volume of $C \setminus \Gm$ the \emph{covolume} of $\Gm$ and denote it by $\Covol(\Gm)$. Analogously to the mixed volume case, we define the \emph{mixed covolume} $ \Covol(\Gamma_1, \Gm_2, \Gm_3,  \ldots, \Gamma_n)$ of $\Gamma_1, \ldots, \Gamma_n$ to be given by the polarization of the covolumes. When $\Delta = \{ x \in \RR^n_{+} : x_1 + \ldots + x_n \ge 1 \} $, we define the $k$-th covolume $\Covol_k (\Gamma) := \Covol(\Gamma, \ldots, \Gamma, \Delta, \ldots, \Delta)$ where there are $k$ copies of $\Gamma$ and $n-k$ copies of $\Delta$ \cite[p.1983]{R13b}.

\begin{remark}

As in \cite{KK}, more generally than $C=\overline{\RR^n_{+}}$, one can let $C$ be a dimension $n$ closed strictly convex cone with apex at the origin $0$ in the Euclidean space $\RR^n$. In this paper, we will assume $C = \overline{\RR^n_{+}}$ since this is the representative case which suffices for our purpose in this paper.  For more general $C$, one may use approximation by those $C$ being rational polyhedral cones and then use the associated affine toric varieties as in \cite[Proof of Theorem 10.4]{KK13} toward results in Section 5 (cf. \cite[p.218]{T04}).

\end{remark}


\section{ Plurisubharmonic functions : Cegrell and toric }\label{cobdd}

 In this section, we collect and summarize the background knowledge on the plurisubharmonic functions used in this article: those in the Cegrell classes on one hand, and toric psh functions on the other hand. The former generalizes the psh functions with isolated singularities.

\subsection{Plurisubharmonic functions in the Cegrell classes and their mixed Monge-Amp\`ere masses}

 Analogously to the previous section, we can define the mixed Monge-Amp\`ere mass (see e.g. \cite{D93}, \cite{R03})  $$m(\vp_1, \vp_2, \vp_3, \ldots, \vp_n) (p)$$ at a point $p \in \Omega$ of psh functions $\vp_1, \ldots, \vp_n$ to be the polarization of  $ (dd^c \vp)^n (p)$ when these are defined, for example when $\vp_1, \ldots, \vp_n$ have isolated singularities at $p$.

 The complete generality of the definition of $(dd^c \vp)^n$ was investigated by \cite{C04} (see also \cite{B06}) and as we will see, this is also where the above mixed Monge-Amp\`ere mass can be defined.
 We recall from \cite{C04}, \cite{DH} (to which we refer for more details and notations) the classes $\EE_0, \FF, \EE, \tilde{\EE}$ of psh functions which generalize those with isolated singularities.

  Let $\Omega \subset \CC^n$ be a bounded hyperconvex domain (i.e. a domain with a negative psh exhaustion function). First we define $$ \EE_0 (\Omega) = \{ \vp \in \PSH^- (\Om) : \lim_{z \to \partial \Om} \vp(z) = 0 \text{\qa and } \int_\Om  (dd^c \vp)^n < \infty \} .$$

 Second, $\FF (\Om)$ is the set of negative psh functions $\vp$ for which there exists a decreasing sequence $\vp_p$ in $\EE_0 (\Omega)$ with a uniform bound of their Monge-Amp\`ere masses $\sup_{p \ge 1} \int_\Om (dd^c \vp_p)^n < \infty$ such that $\vp$ is the decreasing limit of $\vp_p$ at every point.

 Third, $\EE(\Om)$ is the set of negative psh functions $\vp$ such that for every relatively compact $K \Subset \Om$, there exists $\vp_K$ in the class $\FF(\Om)$ such that $\vp = \vp_K$ on $K$. It is the biggest subset of $\PSH^- (\Om)$ where the Monge-Amp\`ere operator is well defined \cite{C04}.  Since our interest is naturally on the behaviour of a psh function near one point, there is no loss of generality only to consider negative psh functions such as those in $\EE$.

 On the other hand, one can also use the class $\EEt$ introduced in \cite{DH} : $\tilde{\EE}(\Om)$ is the set of (not necessarily negative) psh functions $\vp$ which is locally equal to the sum of a function from $\EE(\Om)$ and a function from $C^{\infty} (\Om)$. It is the biggest subset of $\PSH(\Om)$ where Monge-Amp\`ere operator is locally well defined \cite{DH}. If $\vp \in \PSH(\Om)$ has isolated singularity at $0 \in \Om$, then $\vp \in \tilde{\EE} (\Om)$. Also we have $\tilde{\EE} \supset \EE \supset \FF \supset \EE_0$.

  The singularities of a psh function in the Cegrell classes are typically even more complicated than the case of a psh function with isolated singularities. For instance,   the set of points $p$ where the residual Monge-Amp\`ere mass $(dd^c \vp)^n (p)$ is positive, can have a limit point, as the next example shows already in dimension $1$ (there is a similar example in all dimensions). 

\begin{example}\label{limit}

$u(z)=\log|z| + \sum_{k \ge 2} 2^{-k} \log|z-\frac{1}{k} |$. 
\end{example}

\begin{remark1}

 As seen from this example, the comments on discreteness in \cite[p.171, Remark]{C04}, \cite[end of Section 3.2]{DH} are slightly inaccurate. See (\ref{discrete}) below instead.  
\end{remark1}

\noindent For another example of bad behavior, see \cite[Example 2.1]{ACP15} for a Cegrell psh function which takes the value $-\infty$ on a dense set of points.

  Now \cite[Definition 4.3]{C04} defines  the mixed Monge-Amp\`ere current   $dd^c \vp_1 \wedge \ldots \wedge dd^c \vp_n$ as a Radon measure for  $\vp_1, \ldots, \vp_n \in \EE$, and the mixed Monge-Amp\`ere mass $m(\vp_1, \vp_2, \ldots, \vp_n) (p)$ at $p$ is defined to be  the point mass  $(dd^c \vp_1 \wedge \ldots \wedge dd^c \vp_n) (p)$. Higher Lelong numbers of $\vp$ in $\EE$ is defined by the mixed Monge-Amp\`ere mass:

  $$ \ell_k (\vp, p) = \int_{\{p\}} (dd^c \vp)^k \wedge (dd^c \log \abs{z} )^{n-k} = m(\vp, \ldots, \vp, \log \abs{z}, \ldots, \log \abs{z} )   $$

\noi where $\vp$ is repeated $k$ times in the last item, of course.

 Finally we remark that Conjecture~\ref{dem} makes sense since the approximating function $\vp_m$ has at worst isolated singularities (and thus $\ell_k (\vp_m, 0)$ is defined) by the following  

\begin{proposition}\label{discrete}

 Let $\vp \in \EE(\Om)$.  For each $m > 0$,  the set $T_m := \{ x \in \Omega: \JJ (m \vp)_x  \neq \OO_{\Om, x}  \}  $ is discrete. It follows that for every $x \in \Om$, the stalk of the multiplier ideal $\JJ (m \vp)_x$ is either equal to $\OO_{\Om, x}$ or $\mathfrak{m}$-primary where $\mathfrak{m} = \mathfrak{m}_x$.

\end{proposition}
\begin{proof}

From \cite[Corollary 2.2]{DH}, we know $\ell_n (\vp,x) \ge \ell_1(\vp,x)^n$. Hence we have
 
 $$T_m \subset \{ x \in \Omega: \ell_1(\vp,x) \ge 1 \} \subset  \{ x \in \Omega: \ell_n(\vp,x) \ge 1 \}  .$$
 
\noi The last set is discrete since the Monge-Amp\`ere  mass  is locally finite. Therefore $T_m$ is discrete. It is clear that the last sentence follows from this discreteness. 

\end{proof}

 \subsection{ Toric psh functions and indicators }
  We will follow \cite[Section 3]{R13b} (see also \cite{G}) for the material in this subsection. Since our interest is on the singularity of a psh function at one point, we may assume that the domain of psh functions in the following is the unit polydisk $\DD^n$ centered at $0 \in \CC^n$.

 It is well known that given a toric psh function $\vp(z_1, \ldots, z_n)$ on $\DD^n$, a convex function $\widehat{\vp}$ on $\RR^n_{-}$, increasing in each variable, is associated so that $\vp(z_1, \ldots, z_n) = \widehat{\vp} ( \log\abs{z_1}, \ldots, \log\abs{z_n})$.    Following \cite{L94}, the function $\widehat{\vp}$ was called the \emph{convex image} of $\vp$ in \cite{R13b}.

 For $\vp \in \PSH^- (\DD^n)$, there exists a toric psh function $\Psi_\vp$ on $\DD^n$ called the \emph{indicator} (at $0 \in \DD^n$) (see \cite{LR}) which computes the Kiselman number $\nu_{\vp,a}$ of $\vp$ at $0 \in \DD^n$ in the direction $a=(a_1, \ldots, a_n) \in \RR^n_{+}$ in the following sense:

 $$  \Psi_\vp (z) = - \nu_{\vp,a} , \quad  a = -(\log \abs{z_1}, \ldots, \log \abs{z_n}),$$
 for $z$ outside the coordinates hyperplanes.

 The convex image $\widehat{{\Psi}_\vp}$ determines a closed convex set $\Gamma_\vp$ in the closure of ${\RR^n_{+}}$ defined by

 $$ \Gamma_\vp = \{ a \in \overline{\RR^n_{+}}    :    \inp{a,t}  \le   \widehat{{\Psi}_\vp} (t),   \forall  t \in \RR^n_{-}   \} $$

\noi which is called the \emph{indicator diagram} of $\vp$ (at $0 \in \DD^n$).  It generalizes the Newton polyhedron associated to a monomial ideal $\mfa$ via the toric psh function corresponding to $\mfa$.

\begin{remark}

 When $\vp$ is toric psh, \cite[Lemma 1.19]{G} defines a similar object as the \textit{Newton convex body of the homogenization of the associated concave function} of $\vp$.

 \end{remark}

  The convex image of the indicator $\widehat{\Psi_\vp}$ satisfies the homogeneity $\widehat{\Psi_\vp} (ct) = c \widehat{\Psi_\vp} (t) $ for every $c > 0$.  We will call these functions \emph{abstract indicators} (see \cite{R01}, also \cite[p.1981]{R13b}).     In particular, it is the restriction to $\RR^n_{-}$ of the support function of some closed convex subset $\Gamma \subset \overline{\RR^n_+}$ (cf. \cite[Theorem 2.2.8]{H94}).

\section{Proof of Theorem~\ref{dem toric}   }

In this section, we will prove Theorem~\ref{dem toric}. We first need the following relation.

\begin{theorem}\label{r13}

 Suppose that $\vp(z)$ is a toric psh function in the Cegrell class $\EE = \EE(\DD^n)$. Then its indicator diagram $\Gamma_\vp$ at $0 \in \DD^n$ has finite covolume (in $\overline{\RR^n_{+}}$) and

\begin{equation}\label{covol}
 \ell_k (\vp, 0) = n! \Covol_k (\Gamma_\vp)
\end{equation}
\noi  for $k= 1, \ldots, n$.

\end{theorem}

This generalizes \cite[Corollary 3.3]{R13b} from the isolated singularity case, in which case the indicator diagram is moreover cobounded.

\begin{proof}

Since $\vp \le \Psi_\vp + O(1)$, the indicator $\Psi_\vp$ is also in the Cegrell class $\EE$. Let us construct a sequence of abstract indicators with isolated singularities $$\Psi_N=\max\{\Psi_\vp, N\log|z_1|,\ldots,N\log|z_n|\}$$ which decreases to $\Psi_\vp$.  Let $\Gamma_N$ be the indicator diagram of $\Psi_N$. Then $\ell_n (\Psi_N) = n!  \Covol (\Gamma_N)$ holds by \cite[(20) and Corollary 3.3]{R13b}.  Since  $\Psi_\vp$ is in the Cegrell class and $\Psi_N \ge \Psi_\vp$, it follows that $\ell_n (\Psi_N)$ is bounded above by $\ell_n(\Psi_\vp)$.

  As we have the convergence $\Psi_N \to \Psi_\vp$, so do we have the indicator diagram $\Gamma_N \to \Gamma_\vp$ in the sense that the covolume of $\Gamma_\vp$ is the limit of the covolumes of $\Gamma_N$ which is finite and equal to $\ell_n (\Psi_\vp) $.

 Now it remains to show that $\ell_n (\vp)= \ell_n (\Psi_\vp)$. This will complete the proof of (2) for general $k$ as well by applying the polarization identity \eqref{polari} for multilinear symmetric functions $\mu (s_1, \ldots, s_n)$ \cite[(21)]{R13a} :
  to both $k$-th Lelong number and $k$-th covolume (see \cite[(21) and Corollary 3.3]{R13b}).

 Consider the family given by 
 
\begin{equation}\label{mtrans} 
\vp_m(z) = \frac{1}{m} \vp (|z_1|^m,\ldots, |z_n|^m).
\end{equation}
 The sequence $\vp_m$ increases to a function whose upper semicontinuous regularization coincides with the indicator $\Psi_\vp$ of $\vp$ \cite[(2.1)]{R17}   (see Remark~\ref{vpm}) :

\begin{equation}\label{indi}
 \Psi_\vp (z_1, \ldots, z_n) = \limsup_{y \to z} \lim_{m \to \infty} \frac{1}{m} \vp (\abs{y_1}^m, \ldots, \abs{y_n}^m) .
\end{equation}

 Therefore the measures $(dd^c \vp_m)^n$ converge to $(dd^c \Psi_\vp)^n$ by Cegrell's theorem on the continuity of Monge-Amp\`ere operator with respect to increasing sequences \cite[Remark, p.175]{C04} and \cite{C12}. Then  we have

 $$ \ell_n (\vp, 0) = (dd^c \vp)^n (0)= \lim_{m\to\infty} (dd^c \vp_m)^n (0) \le  (dd^c \Psi_\vp)^n (0) = \ell_n (\Psi_\vp, 0) ,$$ where the second equality is from Lemma~\ref{vpmm} and the inequality   is due to the convergence of the measures $(dd^c \vp_m)^n$.  The other direction of the inequality,  $  \ell_n (\vp, 0) \ge  \ell_n (\Psi_\vp, 0)  $, follows from the relation $\vp \le \Psi_\vp + O(1)$ and   \cite[Lemma~4.1]{ACCP} which is an extension of Demailly's comparison theorem \cite[Proposition~5.9]{D87} to functions of the class $\EE$.

\end{proof}

\begin{lemma}\label{vpmm}
We have  $(dd^c \vp)^n (0)= (dd^c \vp_m)^n (0)$ for every $m \ge 1$.

\end{lemma}

\begin{proof}

 Consider $T_m: \DD^n \to \DD^n$, a ramified covering of degree $mn$ of the unit polydisk given by  $ z=(z_1, \ldots, z_n) \mapsto (z_1^m, \ldots ,z_n^m)$.   Let $\vp_\epsilon \; (\ep > 0)$ be $C^{\infty}$ regularizations of $\vp$ (which we can take as the standard regularization as in e.g.  $\varphi_\epsilon(z)=\int \varphi(z+\zeta)\chi_\epsilon(\zeta)\, dV(\zeta)$).

 Then for almost all small radii $\delta>0$ (for which there is no mass of $(dd^c \vp)^n$ on the boundary of $\DD^n_\delta$),  the Monge-Amp\`ere mass of $\vp_\epsilon$ in the polydisk  $\DD^n_\delta$ with radii $\delta$  converges to that of $\vp$  by the fundamental properties of the Cegrell classes  \cite{C04}, \cite[Theorem 1.1]{B06}, hence we get, as $\ep \to 0$, 

 $$ \int_{\DD^n_{\delta^m}} (dd^c \vp_\epsilon)^n \to  \int_{\DD^n_{\delta^m}} (dd^c \vp)^n .$$
 
 \noindent Now let $\varphi_{\epsilon,m}(z) =\frac1m \varphi_\epsilon(|z_1|^m,\ldots, |z_n|^m)$ as in \eqref{mtrans}. We have

 $$ \int_{\DD^n_{\delta}} (dd^c \vp_{\epsilon, m})^n  \to  \int_{\DD^n_{\delta}} (dd^c \vp_{m})^n $$

\noindent as $\ep \to 0$, since $\varphi_\epsilon$ decrease to $\varphi$ and thus $\varphi_{\epsilon,m}$ decrease to $\varphi_m$. On the other hand,   we also  have
 $$ \int_{\DD^n_{\delta^m}} (dd^c \vp_\epsilon)^n  = \int_{\DD^n_{\delta}} (dd^c \vp_{\epsilon, m})^n   $$
from the coordinate change formula and the invariance of the Monge-Amp\`ere operator with respect to holomorphic mappings for smooth functions. Therefore when we let $\delta \to 0$, we obtain the equality to be shown.
\end{proof}

\begin{remark}\label{vpm}

  For the sequence in \eqref{indi}, its $L^1_{\loc}$ convergence was given in \cite[Theorem 8]{R00}. Also note that we indeed need upper semicontinuous regularization (over the zero set $(z_1 \ldots z_n = 0) \subset \CC^n$)  in view of an example $\vp(z_1, z_2)= -(-\log|z_1|)^{1/2}$ whose indicator is identically zero while $\vp_m(0,z_2)=-\infty$. (But outside the set $(z_1 \ldots z_n = 0)$, we do not need $\limsup_{y \to z}$ by convexity argument. )

   Note here that $\vp(z_1, z_2) \in \EE(\DD^2)$ since $\vp$ is the limit of a decreasing sequence of functions $u_\epsilon \in \EE(\DD^2)$ with uniformly bounded Monge-Amp\`ere masses near $0$, for example taking $u_\epsilon=- (-\log(|z_1|+\epsilon|z_2|))^{1/2}$.

\end{remark}

\begin{remark}

  In \cite[p.1980, Remarks]{R13b}, it is mentioned that $\ell_k (u) = \ell_k (g_u)$ is not known for $u$ in the Cegrell classes where $g_u$ is the greenification of $u$ (see \cite[p.1979]{R13b}).  Since $g_u = \Psi_u$ by \cite[Theorem 3.1]{R13b}, the equality  $\ell_n (\vp) = \ell_n (\Psi_\vp)$  in our proof confirms $\ell_k (u) = \ell_k (g_u)$  for toric $u$ in the Cegrell classes.

\end{remark}

\begin{remark}

   In general, it is not true that a toric psh function with the indicator diagram of finite covolume is in the Cegrell classes. Consider the following example by Kiselman \cite{Ki83}:  $u(z_1, z_2) = \sqrt{ \abs{ \log \abs{z_1} }    }   (\abs{z_2}^2 -1 ) $  on $\DD^2$.
(See also \cite[Example 3.32]{GZe}.) It has zero Lelong numbers at every point and its indicator diagram has zero covolume. But it has infinite Monge-Amp\`ere mass in any neighborhood of the origin.   Nevertheless, any finite covolume convex subset $\Gamma$ of $\overline{\RR^n_{+}}$ is the Newton diagram of some toric psh function $\vp$ of the Cegrell class $\EE$, see Lemma~\ref{indic}.

\end{remark}

  Actually, in the proof of Theorem~\ref{r13}, we have never used the fact that we were computing precisely the higher Lelong numbers of $\vp$  (i.e. just the mixed Monge-Amp\`ere mass of the toric Cegrell psh functions $\vp$ and $\log|z|$). Therefore, the same arguments give us

\begin{theorem}\label{r13 general}

 Suppose that $\vp_1,\ldots,\vp_n$ are toric psh function in the Cegrell class $\EE(\DD^n)$. Then
\begin{equation}\label{covol general}
 m(\vp_1,\ldots,\vp_n)= n! \Covol(\Gamma_{\vp_1},\ldots,\Gamma_{\vp_n}).
\end{equation}

\end{theorem}

\begin{remark}    If  $\psi\in\EE$, the mixed Monge-Amp\`ere mass $ m(\vp,\psi,\ldots,\psi)$ and the covolume $\Covol(\Gamma_\vp,\Gamma_\psi,\ldots,\Gamma_\psi)$ can be computed for any psh $\vp$. However, the relation (\ref{covol general}) need not be true if $\vp\not\in\EE$ (the Cegrell class), even if it is toric. Indeed, for example when $n=2$, if $\vp(z_1,z_2)= \log|z_1|$ and $\psi(z_1,z_2)=
\max\{-\sqrt{ \abs{ \log \abs{z_1} } }, \log \abs{z_2}\} $, then we have
$m(\vp,\psi)=1$ while $\Covol(\Gamma_\vp,\Gamma_\psi)=0$. In particular, this example shows that even in the toric case, one can have psh functions $\vp$ and $\psi$ such that $m(\vp,\psi) > 0$ where
 $\psi$ has vanishing first Lelong numbers (and $\vp$ not in the Cegrell class). 

\end{remark}

\medskip

\begin{proof}[Proof of Theorem~\ref{dem toric}]

 We have the following equality of the higher Lelong number and covolume for a toric Cegrell psh function from Theorem~\ref{r13}:

\begin{equation}\label{LC1}
\ell_k (\vp, 0) = n! \Covol_k (\Gamma_\vp).
\end{equation}

Now the Demailly approximation $\vp_m$ has analytic singularities given by the multiplier ideals $\JJ(m\vp)$ which are monomial ideals as $\vp$ is toric psh. Thus we have

\begin{equation}\label{LC2}
\ell_k (\vp_m, 0) = \frac{1}{m^n}\, e (\JJ (m \vp), \ldots, \JJ(m\vp), \mathfrak{m}, \ldots, \mathfrak{m}) = \frac{n!}{m^n} \Covol_k (N(\JJ (m\vp))),
\end{equation}
where the first equality is given by \cite[Lemma~2.1]{D09} and polarization from the case $k=n$. In the mixed multiplicity $e (\JJ (m \vp), \ldots, \JJ(m\vp), \mathfrak{m}, \ldots, \mathfrak{m})$, the ideals $\JJ(m\vp)$ and $\mathfrak{m}$ are repeated $k$ times and $n-k$ times, respectively.

 In \eqref{LC2}, $N(\JJ (m\vp))$ denotes the Newton polyhedron of the monomial ideal $\JJ(m\vp)$.  The second equality is well known when $k=n$ (e.g. as a special case of \cite[Theorem 8.12]{KK}, see also \cite[Remark 7.9]{KK}) and then in general by polarization. Now we have 

 \begin{lemma}\label{multi}\cite[Theorem A]{G}, \cite[Proposition 3.1]{R13b}

Let $\vp = \vp(z_1, \ldots, z_n)$ be a toric psh function. Then the multiplier ideal $\JJ(\vp)$  is a monomial ideal such that $z_1^{a_1} \ldots z_n^{a_n} \in \JJ( \vp)$ if and only if $(a_1 + 1, \ldots, a_n +1) \in \interior \Gamma_{\vp} $.

\end{lemma}

 Applying Lemma~\ref{multi} to $m\vp$ as $m \to \infty$, we see that the Newton polyhedra $N(\JJ (m\vp))$ divided by $m$ approximates $\Gamma_\vp$ from outside in the sense that $$ \Gamma_\vp = \bigcap_{m \ge 1} \frac{1}{m} N(\JJ (m\vp))  .$$

\noi  Thus in view of \eqref{LC1} and \eqref{LC2}, the convergence $\ell_k (\vp_m) \to \ell_k (\vp)$ follows from the convergence  $\frac{1}{m^n} \Covol_k (N(\JJ (m\vp))) \to  \Covol_k (\Gamma_\vp)$.

\end{proof}

\begin{remark}\label{rr1}

 Thanks to Theorem~\ref{r13 general}, the arguments in the proof of Theorem~\ref{dem toric} also give the generalized version of Theorem~\ref{dem toric} in Remark~\ref{rr0} for general mixed Monge-Amp\`ere masses. 
 
\end{remark}

\begin{remark}\label{rr2}
On the other hand, it is not hard to see that one cannot simply reduce the general case of Conjecture~\ref{dem}~(a) to the case $k=n$ just using the polarization identity~\eqref{polari} since the $m$-th Demailly approximation $(\vp + \psi)_m$ is not equal to $\vp_m + \psi_m$ (up to $O(1)$). In fact $(\vp + \psi)_m \le \vp_m + \psi_m  + O(1)$ due to subadditivity of multiplier ideals.

\end{remark}

\section{Applications to convex geometry }

In this section, we first extend the basic theory of mixed covolumes in Section~\ref{cobdd} from the cobounded case to the cofinite case.

\subsection{Mixed covolumes: generalization to the cofinite case}

 Let $C = \overline{\RR^n_{+}}$ again in this section.

\begin{lemma}\label{indic}

 Let $\Gamma$ be a $C$-convex region with finite covolume.  Then there exists a toric psh function $\vp$ in the Cegrell class $\FF = \FF(\DD^n)$ such that $\Gamma$ is equal to the indicator diagram of $\vp$ at $0 \in \DD^n$. Moreover, $\vp$ can be taken as an indicator.  In particular,

\begin{equation}
 \ell_k (\vp, 0) = n! \Covol_k (\Gamma)
\end{equation}
\noi  for $k= 1, \ldots, n$.

\end{lemma}

\begin{proof}

The following argument generalizes \cite[Example 3.9]{R13a}. Given cofinite $\Gm \subset C = \overline{\RR^n_{+}}$, certainly we can approximate $\Gm$ by cobounded $C$-regions. For example, take the support function of $\Gm$ which is the convex image of an indicator function $\Psi$.  We will show that we can take $\vp = \Psi$.  Take again the same sequence $\Psi_N$ as in the proof of Theorem~\ref{r13}:

 $$\Psi_N (z_1, \ldots, z_n) =\max\{\Psi(z_1, \ldots, z_n), N\log|z_1|,\ldots,N\log|z_n|\}$$ decreasing to $\Psi$.  Let $\Gamma_N$ be the indicator diagram of $\Psi_N$, which is cobounded since $\Psi_N$ has isolated singularity. Then   $\ell_n (\Psi_N) = n!  \Covol (\Gamma_N)$ holds by \cite[(20) and Corollary 3.3]{R13b}.   Since $\Psi_N \ge \Psi$, we have $\Gamma_N \subset \Gamma$.   From \cite[Proposition 1,(d)]{R00},  we have $\int_\Om (dd^c \Psi)^n = \ell_n (\Psi_N) = n!  \Covol (\Gamma_N) \le n! \Covol (\Gamma) $.

 On the other hand, $\Psi_N$ has value zero on the boundary $\partial \Om$. Thus $\Psi_N \in \EE_0$ with uniform bound less than the finite covolume of given $\Gamma$.  Since $\Psi$ is the decreasing limit of $\Psi_N$'s, we have $\Psi \in \FF$.  Now apply Theorem~\ref{r13} to $\vp = \Psi$ and we are done.

\end{proof}

 Let $\Gm_1, \Gm_2$ be $C$-convex regions. Then for $\ld_1, \ld_2  > 0$,  $\ld_1 \Gm_1 + \ld _2 \Gm_2 = \{ \ld_1 x_1 + \ld_2 x_2  : x_1 \in \Gm_1, x_2 \in \Gm_2 \} $  is also a $C$-convex region by \cite[Proposition 10.3]{KK13}.

 \begin{proposition}

 If $\Gm_1$ and $\Gm_2$ are cofinite, then $\ld_1 \Gm_1 + \ld _2 \Gm_2$ is cofinite.

 \end{proposition}

\begin{proof}

 From the basic properties of support functions (cf. \cite[p.75]{H94}),  the indicator function given by Lemma~\ref{indic} for $\ld_1 \Gm_1 + \ld _2 \Gm_2$ is $\ld_1 \psi_1 + \ld_2 \psi_2$.   By the same lemma, $\Covol (\ld_1 \Gm_1 + \ld_2 \Gm_2 ) = \ell_n (\ld_1 \psi_1 + \ld_2 \psi_2, 0)$ which is finite since $\ld_1 \psi_1 + \ld_2 \psi_2$ is in the Cegrell class $\FF$.

\end{proof}

\begin{theorem}

 Let $\Gamma_1, \ldots, \Gamma_n$ be cofinite $C$-convex regions. Then the function

\begin{equation}
 P(\ld_1,\ldots, \ld_n) = \Covol (\ld_1 \Gamma_1 + \ldots + \ld_n \Gamma_n )
\end{equation}

\noindent is a homogeneous polynomial of degree $n$ in $\ld_1, \ldots, \ld_n$.

\end{theorem}

\begin{proof}

 Let $\psi_1, \ldots, \psi_n$ be the indicator functions given by Lemma~\ref{indic}. Then $ (dd^c (\ld_1 \psi_1 + \ldots + \ld_n \psi_n) )^n $  is expanded as a homogeneous polynomial of degree $n$ in $\ld_1, \ldots, \ld_n$ with coefficients given by mixed Monge-Amp\`ere currents such as $(dd^c \psi_1)^n$ etc. By taking the mass at the point $0 \in \DD^n$, the theorem is proved.

\end{proof}

\subsection{Proof of Theorem~\ref{af}}

  Let $\vp_1, \ldots, \vp_n \in \EE(\Omega)$. By \cite[p.521]{B06}, belonging to $\EE$ is a local property, thus to show the inequality at $p$, we may assume that the domain $\Omega$ equals $B$, a ball centered at $p$. Moreover by definition of $\EE$ in \cite{C04}, there is a neighborhood $U$ of $p$ where $\vp_j = \psi_j$ on $U$ for some $\psi_j \in \FF(B)$ for every $1 \le j \le n$. Since the mass at $p$ only depends on the restriction of the Monge-Amp\`ere current to a neighborhood of $p$,  we may assume $\vp_1, \ldots, \vp_n \in \FF(B)$.

  By definition of the class $\FF$, for each $\vp_j$, there exists a sequence $\vp_{j,i}$ in $\EE_0$ converging to $\vp_j$. Applying \cite[Lemma 5.4]{C04} for $T = -h ( dd^c \vp_{3,i}) \wedge \ldots \wedge (dd^c \vp_{n,i})$, $p=q=1$, we get 

\begin{align*}
   &\left(\int_B -h dd^c \vp_{1,i} \wedge dd^c \vp_{2,i} \wedge dd^c \vp_{3,i} \wedge \ldots \wedge dd^c \vp_{n,i} \right)^2 \\
   \le &   \left( \int_B -h dd^c \vp_{1,i} \wedge dd^c \vp_{1,i} \wedge dd^c \vp_{3,i} \wedge \ldots \wedge dd^c \vp_{n,i} \right) \times \\
    & \quad \left( \int_B -h dd^c \vp_{2,i} \wedge dd^c \vp_{2,i} \wedge dd^c \vp_{3,i} \wedge \ldots \wedge dd^c \vp_{n,i} \right).
\end{align*}

 Now when we let $i \to \infty$ and apply \cite[Prop. 5.1]{C04}, we get the corresponding inequality for $\vp_1, \ldots, \vp_n$. As in the proof of \cite[Corollary 5.7]{C04} and \cite[Lemma 2.1]{DH}, let $h = h_r = \max ( \frac{1}{r} \log \abs{z}, -1 ) \in \EE_0 (B)$.  As $r \to \infty$, $h_r$ converges to the negative of the characteristic function of the point $0 \in \CC^n$.  By the monotone convergence theorem, Theorem~\ref{af} is proved. \qed

 \begin{remark}\label{dh21}

 In Theorem~\ref{af}, choosing each of $\vp_3,\ldots,\vp_n$ equal to either $\vp_1$ or $\vp_2$ recovers \cite[Lemma 2.1]{DH}. Note that
 in its proof provided in \cite{DH}, replacing $\vp$ with $\tilde{\vp} := \max\{\vp-C, p\log|z|\}$ was done without justifying that $\ell_k(\vp)$ remains the same. Since Conjecture~\ref{dem} (b) is known to be true for functions dominated by  $ p \log \abs{z}$ with $p > 0$ \cite{W05}, \cite{R16}, this is in fact as nontrivial as Conjecture~\ref{dem} (b) to check. The proof of Theorem~\ref{af} makes up necessary details. 

 \end{remark}

  Finally, we get an analytic proof of the Alexandrov-Fenchel inequality for mixed covolumes, i.e. Corollary~\ref{AFM}, by applying Lemma~\ref{indic} to Theorem~\ref{af}.

\footnotesize

\bibliographystyle{amsplain}

\qa

\normalsize

\noi \textsc{Dano Kim}

\noi Department of Mathematical Sciences \& Research Institute of Mathematics

\noi Seoul National University, 08826  Seoul, Korea

\noi Email address: kimdano@snu.ac.kr

\qa
\qa

\noi

\noi \textsc{Alexander Rashkovskii}

\noi Tek/nat, University of Stavanger, 4036 Stavanger, Norway

\noi Email address: alexander.rashkovskii@uis.no
\noi

\end{document}